\documentclass{ifacconf}

\usepackage{graphicx}      
\usepackage{natbib}        

\usepackage{amsmath,amssymb,amsfonts}

\usepackage[T1]{fontenc}
\usepackage{comment}

\usepackage{tikz}
\usepackage{pgfplots}
\usepackage{subfig}
\usepackage{caption}
\usepackage{siunitx}

\newcommand{\cN}{\mathcal{N}}

\newcommand{\cG}{\mathcal{G}}

\newcommand{\cD}{\mathcal{D}}

\newcommand{\cF}{\mathcal{F}}

\usetikzlibrary{arrows}
\usetikzlibrary{fit}\usetikzlibrary{calc}
  \pgfdeclarelayer{background}
  \pgfsetlayers{background,main}

\newtheorem{Prop}{Proposition}[section]
\newtheorem{Lem}[Prop]{Lemma}
\newtheorem{Thm}[Prop]{Theorem}

\newtheorem{Def}[Prop]{Definition}
\newtheorem{Rem}[Prop]{Remark}

\newcommand{\RN}[1]{\uppercase\expandafter{\romannumeral#1}}
\newcommand{\eps}{\varepsilon}

\newcommand{\N}{{\mathbb{N}}}

\newcommand{\R}{{\mathbb{R}}}

\newcommand{\cTT}{{\mathbb T}^{n,q}_h}
\newcommand{\cCC}{C([-h,\infty),\R^n)}
\newcommand{\cLL}{L_{\rm loc}^\infty(\R_{\ge 0},\R^q)}

\DeclareMathOperator{\esssup}{ess\, sup}

\newcommand{\setdef}[2]{\left\{\, #1 \left|\, \vphantom{#1} #2\right.\right\}}
\newcommand{\ddt}{\tfrac{\text{\normalfont d}}{\text{\normalfont d}t}}

\DeclareMathOperator{\sgn}{sgn}

\DeclareMathOperator{\sat}{sat}

{\left(\begin{smallmatrix}}
{\end{smallmatrix}\right)}

{\left[\begin{smallmatrix}}
{\end{smallmatrix}\right]}

\renewcommand{\theenumi}{{\rm(\roman{enumi})}} 

\begin{document}
\begin{frontmatter}

\title{An improved input-constrained funnel controller for nonlinear systems\thanksref{footnoteinfo}} 

\thanks[footnoteinfo]{Funded by the Deutsche Forschungsgemeinschaft (DFG, German
Research Foundation) -- Project-ID 544702565.}

\author[First]{Thomas Berger} 

\address[First]{Martin-Luther-Universit\"at Halle-Wittenberg, Institut f\"ur Mathematik, Theodor-Lieser-Str.~5, 06120~Halle (Saale), Germany (e-mail: thomas.berger@mathematik.uni-halle.de).}

\begin{abstract}                
We present an improvement of a recent funnel controller design for uncertain nonlinear multi-input, multi-output systems modeled by higher order functional differential equations in the presence of input constraints. The objective is to guarantee the evolution of the tracking error within a performance funnel with prescribed desired shape for the case of inactive saturation. Compared to its precursor, controller complexity is significantly reduced, much fewer design parameters are involved and simulations exhibit a superior performance.
\end{abstract}


\end{frontmatter}

\section{Introduction}\label{Sec:Intr}
%

The purpose of this paper is to present an improvement of a controller design recently published by the author in~\cite{Berg24}. For an extensive literature survey on the topic we refer to the aforementioned work and do not repeat it here. In the following we briefly recall the problem statement.\\
We study funnel control for the class of nonlinear multi-input multi-output systems described by the $r$-th order functional differential equation
\begin{equation}\label{eq:Sys}
    \begin{aligned}
      y^{(r)}(t) &= f\big(d(t), T(y,\dot y,\ldots,y^{(r-1)})(t),u(t)\big),\\
      y|_{[-h,0]} &= y^0 \in C^{r-1}([-h,0],\R^m),
    \end{aligned}
\end{equation}
with continuous function $f$ and operator~$T$ (with properties to be specified later) as well as bounded disturbance~$d$ and  initial trajectory~$y^0$~-- all of these parameters are unknown and not available for controller design. Furthermore, the system is subject to input constraints
\begin{equation}\label{eq:IC}
u(t) = \sat(v(t))
\end{equation}
with \textit{known} saturation function $\sat$ and control function~$v$ provided by the to-be-designed controller.\\
The control objective is to ideally achieve a prescribed performance of the tracking error, that is $\|y(t) - y_{\rm ref}(t)\| < \psi(t)$ for some given reference signal $y_{\rm ref}$ and funnel function~$\psi$. However,  since we consider the input constraints to be \textit{hard constraints}, a conflict of objectives arises: As explained in~\cite{Berg24}, it is not possible to simultaneously satisfy the input and output constraints for any given bounded reference signal. Instead, we consider the aforementioned output constraints to be \textit{soft constraints}, i.e., they can be relaxed when needed in order to meet the input constraints. To achieve this, in~\cite{Berg24} a mechanism to dynamically adjust the funnel function~$\psi$ has been presented. The idea is to widen the performance funnel described by~$\psi$ whenever the input saturation is active, so that no constraints are violated. When the saturation becomes inactive the function~$\psi$  reverts to its prescribed shape
exponentially fast. Recently, in a series of papers~\cite{TrakBech22,TrakBech23,TrakBech24}, Trakas and Bechlioulis have developed a similar approach, however for different system classes. In particular, the most recent result in~\cite{TrakBech24} requires the system to be input-to-state stable, which however is often not satisfied in practical applications, see e.g.\ the simple example in Section~\ref{Sec:Sim}. While it is further shown in~\cite{TrakBech24} that the input-to-state stability assumption can be relaxed to a bounded-input-bounded-state assumption on the internal dynamics (commonly required in funnel control, see e.g.~\cite{BergIlch21}), boundedness of closed-loop signals can then only be guaranteed for those evolving in a certain (unknown) compact set. In the present work we focus on a different system class and do not require either of both assumptions.\\
While in~\cite{Berg24} a chain of~$r$ interconnected funnel functions was used, the improved control design that we present here uses only one funnel function. Therefore, the number of dynamic equations involved in the controller design (and hence its complexity) is significantly reduced. Furthermore, in contrast to~\cite{Berg24}, much fewer controller design parameters are comprised and we are able to prove that the control function~$v$ is always bounded for bounded reference signals.

\textit{Nomenclature.}\ 
In the following let $\N$ denote the natural numbers, $\N_0 = \N \cup\{0\}$, and $\R_{\ge 0} =[0,\infty)$. By $\|x\|$ we denote the Euclidean norm of $x\in\R^n$. For some interval $I\subseteq\R$, some $V\subseteq\R^m$ and $k\in\N$, $L^\infty(I, \R^{n})$ $\big(L^\infty_{\rm loc} (I, \R^{n})\big)$ is the Lebesgue space of measurable, (locally) essentially bounded {functions} $f\colon I\to\R^n$, $W^{k,\infty}(I,  \R^{n})$ is the Sobolev space of all functions
$f:I\to\R^n$ with $k$-th order weak derivative $f^{(k)}$ and $f,f^{(1)},\ldots,f^{(k)}\in L^\infty(I, \R^{n})$, and
 $C^k(V,  \R^{n})$ is the set of  $k$-times continuously differentiable functions  $f:  V  \to \R^{n}$, with $C(V,  \R^{n}) := C^0(V,  \R^{n})$.

\textit{System class.}\
We recall the necessary definitions from~\cite{Berg24}.

\begin{Def}
For $n,q\in\N$ and $h \geq 0$ the set $\cTT$ denotes the class of operators $T\colon \cCC\to \cLL$ with the following properties.
\begin{enumerate}
\item[\textbf{(P1)}]
 $T$ is causal, i.e., for all $\zeta$, $\xi \in \cCC$ and all $t\ge 0$,
\[
\zeta|_{[-h,t]} =\xi|_{[-h,t]} ~~\implies~~ T(\zeta)|_{[0,t]}=T(\xi)|_{[0,t]}.
\]
\item[\textbf{(P2)}]
 $T$ is locally Lipschitz, i.e., for each $t\ge 0$ and all $\xi\in C([-h,t],\R^{n})$, there exist positive constants $c_0, \delta, \tau >0$ such that, for all $\zeta_1,\zeta_2 \in \cCC$ with $\zeta_i|_{[-h,t]} = \xi$
and $\|\zeta_i(s)-\xi(t)\|<\delta$ for all $s\in[t,t+\tau]$ and $i=1,2$, we have
\begin{multline*}
 \underset{s\in [t,t+\tau]}{\esssup}\  \|T(\zeta_1 )(s)-T(\zeta_2) (s)\| \\
 \leq c_0 \sup_{s\in [t,t+\tau]}\|\zeta_1(s)-\zeta_2(s)\|.
\end{multline*}
\item[\textbf{(P3)}]
 $T$ locally maps bounded functions to bounded functions, i.e., for all $\tau>0$ and all $c_1 >0$, there exists $c_2 >0$ such that, for all $\zeta\in C([-h,\tau],\R^n)$,
\[
\sup_{t\in[-h,\tau]}\|\zeta(t)\|\le c_1 ~~\implies~~ \underset{t\in[0,\tau]}{\esssup}\ \|T(\zeta)(t)\| \le c_2.
\]
\end{enumerate}
\end{Def}

Next we recall the sector bound property of~$f$ and~$T$. 

\begin{enumerate}
\item[\textbf{(P4)}] For all $y^0 \in C^{r-1}([-h,0],\R^m)$ there exist $M_1, \ldots, M_{r+1}\in C(\R_{\ge 0}\times\R^p\times\R^m,\R_{\ge 0})$ which are bounded in~$t$, such that for all $t\ge 0$, all $(d,v)\in\R^p\times\R^m$ and all $\zeta_1,\ldots,\zeta_r\in C([-h,t],\R^m)$ with $\zeta_i|_{[-h,0]} = (y^0)^{(i-1)}$ for $i=1,\ldots,r$ we have:
    \begin{multline*}
   \hspace{-0.8cm} \|f(d,T(\zeta_1,\ldots,\zeta_r)(t),v)\| \le M_1(t,d,v) \\
    \hspace{-0.4cm}+ M_2(t,d,v) \|\zeta_1|_{[-h,t]}\|_\infty + \ldots + M_{r+1}(t,d,v) \|\zeta_r|_{[-h,t]}\|_\infty
    \end{multline*}
\end{enumerate}

Note that the functions $M_i$ in (P4) depend on the initial history $y^0$ in~\eqref{eq:Sys}. Furthermore, compared to~\cite{Berg24} we additionally assume that each~$M_i$ is bounded in~$t$, which is required to show that the control signal~$v$ in~\eqref{eq:IC}, generated by the controller, is bounded.\\
Next we recall the system class from~\cite{Berg24}. We stress that the high-gain property of system~\eqref{eq:Sys} required in earlier approaches, see e.g.~\cite{BergIlch21}, is not needed here; this is also different from~\cite{TrakBech24}.  It is not even required that~$f$ depends on~$u$; however, in this case it is possible that the tracking error grows unbounded.

\begin{Def}\label{Def:SysClass}
For $m,r\in\N$ we say that system~\eqref{eq:Sys} belongs to the system class $\cN^{m,r}$, written $(d,f,T)\in\cN^{m,r}$, if $d\in L^\infty(\R_{\ge 0},\R^p)$, $f\in C(\R^p\times \R^q\times\R^m,\R^m)$, $T\in {\mathbb T}^{rm,q}_h$ for some $p,q\in\N$, $h\geq 0$ and $(f,T)$ satisfy property~(P4).
\end{Def}

For the saturation function we require the following.

\begin{enumerate}
\item[\textbf{(P5)}] $\sat\!\in\! C(\R^m,\R^m)$ is bounded and there exists $\theta>0$ such that for all $v\in\R^m$ with $\|v\|\le \theta$ we have $\sat(v) = v$.
\end{enumerate}

We stress that the input saturation function $\sat$ must be known to the controller and it can be viewed as a design parameter, chosen according to the specific requirements of the application at hand. The above property~(P5) allows for a large variety of possible saturations, apart from the standard saturation $\sat_i(v) = v_i$ for $|v_i|\le M$ and $\sat_i(v) = \sgn(v_i) M$ for $|v_i|>M$ for all $i=1,\ldots,m$.

\textit{Control objective.}\
The objective is to design a dynamic output derivative feedback strategy such that, ideally, for any reference signal $y_{\rm ref}\in W^{r,\infty}(\R_{\ge 0},\R^m)$ the tracking error $e = y-y_{\rm ref}$ evolves within a performance funnel
\[
    \cF_\psi := \setdef{(t,e)\in\R_{\ge 0}\times\R^m}{\|e\|<\psi(t)},
\]
see Fig.~\ref{Fig:funnel}, which has a desired shape of the form $\psi_{\rm des}(t) = a e^{-bt} +c$ whenever the saturation in~\eqref{eq:IC} is not active, i.e., $\sat(v(t)) = v(t)$, and the actual funnel boundary $\psi(t)$ is allowed to deviate from this shape and become larger when the saturation is active. The specific value of~$\psi(t)$ should be determined by a dynamic part of the control law.

 \begin{figure}[h]
  \begin{center}
\begin{tikzpicture}[scale=0.35]
\tikzset{>=latex}
  \filldraw[color=gray!25] plot[smooth] coordinates {(0.15,4.7)(0.7,3.3)(4,0.6)(6,1.5)(9.5,0.4)(10,0.333)(10.01,0.331)(10.041,0.3) (10.041,-0.3)(10.01,-0.331)(10,-0.333)(9.5,-0.4)(6,-1.5)(4,-0.6)(0.7,-3.3)(0.15,-4.7)};
  \draw[thick] plot[smooth] coordinates {(0.15,4.7)(0.7,3.3)(4,0.6)(6,1.5)(9.5,0.4)(10,0.333)(10.01,0.331)(10.041,0.3)};
  \draw[thick] plot[smooth] coordinates {(10.041,-0.3)(10.01,-0.331)(10,-0.333)(9.5,-0.4)(6,-1.5)(4,-0.6)(0.7,-3.3)(0.15,-4.7)};
  \draw[thick,fill=lightgray] (0,0) ellipse (0.4 and 5);
  \draw[thick] (0,0) ellipse (0.1 and 0.333);
  \draw[thick,fill=gray!25] (10.041,0) ellipse (0.1 and 0.333);
  \draw[thick] plot[smooth] coordinates {(0,2)(2,1.1)(4,-0.1)(6,-0.7)(9,0.25)(10,0.15)};
  \draw[thick,->] (-2,0)--(12,0) node[right,above]{\normalsize$t$};
  \draw[thick,dashed](0,0.333)--(10,0.333);
  \draw[thick,dashed](0,-0.333)--(10,-0.333);
  \node [black] at (0,2) {\textbullet};
  \draw[->,thick](4,-3)node[right]{\normalsize$c$}--(2.5,-0.4);
  \draw[->,thick](3,3)node[right]{\normalsize$(0,e(0))$}--(0.07,2.07);
  \draw[->,thick](9,3)node[right]{\normalsize$\psi(t)$}--(7,1.4);
  \draw [color=blue,thick,smooth,domain=0.05:10] plot(\x,{4.7*exp(-0.9*(\x-0.05))+0.3});
  \draw [color=blue,thick,smooth,domain=0.05:10] plot(\x,{-4.7*exp(-0.9*(\x-0.05))-0.3});
  \draw[->,thick,color=blue](10,-1.8)node[right]{\normalsize$\psi_{\rm des}(t)$}--(8,-0.3);
\end{tikzpicture}
\end{center}
 \caption{Error evolution in a funnel $\mathcal F_{\psi}$ with boundary $\psi(t)$ and desired shape $\psi_{\rm des}(t)$.}
 \label{Fig:funnel}
 \end{figure}

In contrast to classical funnel control~\cite{BergIlch21,IlchRyan02b,BergLe18a}, the funnel boundary~$\psi$ is not fully prescribed, but widened when necessary in order to meet the input constraints. Nevertheless, the desired ``asymptotic shape'' $\psi_{\rm des}(t) = a e^{-bt} +c$ under inactive saturation can be prescribed by choice of the parameters $a,b,c$.



\section{Funnel control structure}\label{Sec:ConStruc}
%

We introduce the following improved input-constrained funnel controller for systems~\eqref{eq:Sys},~\eqref{eq:IC}.
\begin{equation}\label{eq:ICFC}
\boxed{
\begin{aligned}
    e_1(t) &= e(t) = y(t) - y_{\rm ref}(t),\\
    e_{i+1}(t) &= \dot e_i(t) + k_i e_i(t),\quad i=1,\ldots,r-1,\\
    \dot \psi(t) &= -\alpha \psi(t) + \beta + \psi(t) \frac{\kappa(v(t))}{\|e_r(t)\|},\ \psi(0) = \psi^0,\\
    \kappa(v(t)) &= \|v(t)-\sat(v(t))\|,\\
    k(t) &= \left(1- \frac{\|e_r(t)\|^2}{\psi(t)^2}\right)^{-1},\\
    v(t) &= N\big(k(t)\big) e_r(t)
\end{aligned}
}
\end{equation}
with the controller design parameters
\begin{equation}\label{eq:FC-param}
\boxed{
\begin{aligned}
    &\alpha,\beta>0,\  k_1,\ldots,k_{r-1}>\alpha,\ \psi^0>\tfrac{\beta}{\alpha},\\
    &N\in C(\R_{\ge 0},\R)\ \text{ a surjection}.
\end{aligned}
}
\end{equation}

\captionsetup[subfloat]{labelformat=empty}
\begin{figure*}[h!tb]
\centering
\resizebox{0.75\textwidth}{!}{
   \begin{tikzpicture}[very thick,scale=0.7,node distance = 9ex, box/.style={fill=white,rectangle, draw=black}, blackdot/.style={inner sep = 0, minimum size=3pt,shape=circle,fill,draw=black},blackdotsmall/.style={inner sep = 0, minimum size=0.1pt,shape=circle,fill,draw=black},plus/.style={fill=white,circle,inner sep = 0,very thick,draw},metabox/.style={inner sep = 3ex,rectangle,draw,dotted,fill=gray!20!white}]
 \begin{scope}[scale=0.5]
    \node (sys) [box,minimum size=7ex]  {$y^{(r)}(t)= f\big(d(t), T(y,\dot{y},\dots,y^{(r-1)})(t), u(t)\big)$};
    \node [minimum size=0pt, inner sep = 0pt,  below of = sys, yshift=3ex] {System $(d,f,T)\in  \cN^{m,r}$};
    \node(fork1) [minimum size=0pt, inner sep = 0pt,  right of = sys, xshift=32ex] {};
    \node(end1)  [minimum size=0pt, inner sep = 0pt,  right of = fork1, xshift=5ex] {$\big(y,\ldots,y^{(r-1)}\big)$};

   \draw[->] (sys) -- (end1) node[pos=0.4,above] {};

  \node(FC0) [box, below of = fork1,yshift=-5ex,minimum size=7ex] {{$\begin{aligned}
    e_1(t) &= e(t) = y(t) - y_{\rm ref}(t)\\
    e_{i+1}(t) &= e^{(i)}(t) + k_i e_i(t)\end{aligned}$}};
    \node(FC1) [box, below of = sys,yshift=-5ex,minimum size=7ex] {$\dot \psi(t) = -\alpha \psi(t) + \beta + \psi(t) \tfrac{\kappa(v(t))}{\|e_r(t)\|}$};
    \draw[->] (fork1) -- (FC0) {};
    \node(FC2) [box, below of = FC1,yshift=-4ex,minimum size=7ex] {$v(t) =N\big(k(t)\big) e_r(t)$};
     \node(sat) [box, left of = FC1,xshift=-22ex,minimum size=7ex] {$u(t) = \sat(v(t))$};
   \node(e0) [minimum size=0pt, inner sep = 0pt,  below of = FC0, xshift=-7ex, yshift=5ex] {};
   \node(e1) [minimum size=0pt, inner sep = 0pt,  below of = FC0, xshift=7ex, yshift=5ex] {};
   \draw[->] (e0) |- (FC2) node[pos=0.7,above] {$e_r$};
   \node(phi) [minimum size=0pt, inner sep = 0pt,  below of = e1, yshift=1ex] {$\big(y_{\rm ref},\ldots,y_{\rm ref}^{(r-1)}\big)$};
   \draw[->] (phi) -- (e1) node[midway,right] {};
   \draw[->] (FC2) -| (sat)  node[pos=0.7,right] {$v$};
   \draw[->] (sat) |- (sys) node[pos=0.7,above] {$u$};
   \draw[<-] (FC1) -- (FC0) node[midway,below] {$e_r$};

   \node(f0) [minimum size=0pt, inner sep = 0pt,  below of = FC1, xshift=-2ex, yshift=5.6ex] {};
   \node(f1) [minimum size=0pt, inner sep = 0pt,  below of = FC1, xshift=2ex, yshift=5.6ex] {};
   \node(f2) [minimum size=0pt, inner sep = 0pt,  below of = f0, xshift=0ex, yshift=3ex] {};
   \node(f3) [minimum size=0pt, inner sep = 0pt,  below of = f1, xshift=0ex, yshift=3ex] {};
   \draw[->] (f2) -| (f0)  node[pos=0.75,left] {$v$};
   \draw[->] (f1) -| (f3)  node[pos=0.75,right] {$\psi$};

   \node [minimum size=0pt, inner sep = 0pt,  below of = FC2, yshift=2ex, xshift=-8ex] {Funnel controller~\eqref{eq:ICFC}};
\end{scope}
\begin{pgfonlayer}{background}
      \fill[lightgray!20] (-4.3,-1.8) rectangle (13.1,-7.8);
      \draw[dotted] (-4.3,-1.8) -- (13.1,-1.8) -- (13.1,-7.8) -- (-4.3,-7.8) -- (-4.3,-1.8);
  \end{pgfonlayer}
  \end{tikzpicture}
}
\caption{Construction of the funnel controller~\eqref{eq:ICFC} and its internal feedback loop.}
\label{Fig:funnel-controller}
\end{figure*}

The controller~\eqref{eq:ICFC} is similar to that presented in~\cite{Berg24} with some significant differences: a) the gains $k_1,\dots,k_{r-1}$ are not dynamic, but static, b) instead of a chain of~$r$ dynamically generated funnel functions, only one function is used, and c)~much fewer controller design parameters are involved. Clearly, all three improvements a)--c) decrease the complexity of the controller design. On the other hand, it is not directly clear that the simplifications preserve its feasibility and effectiveness~-- this requires a proof which we provide in Theorem~\ref{Thm:FunCon}.\\
The surjective function~$N$ in~\eqref{eq:FC-param} accommodates for possibly unknown control directions. A typical choice for~$N$ would be $N(s) = s \sin s$. For more details see also~\cite[Rem.~1.8]{BergIlch21}.
The structure of the controller~\eqref{eq:ICFC} is illustrated in Fig.~\ref{Fig:funnel-controller}. Note that the division by $\|e_r(t)\|$ in the differential equation for $\psi$ in~\eqref{eq:ICFC} does not cause any numerical issues upon controller implementation, as $\kappa(v(t))=0$ whenever $\|e_r(t)\|<\delta$ for some $\delta>0$. Further note that~\eqref{eq:ICFC} requires the instantaneous signals of the tracking error derivatives $\dot e(t),\ldots, e^{(r-1)}(t)$, which are typically estimated in practice and sensitive to noise. This issue can be treated by using a pre-compensator or filter, see e.g.~\cite{Lanz22a,DennScha25}.
One difference to its precursor presented in~\cite{Berg24} is that for~\eqref{eq:ICFC} it is not directly clear in which performance funnel the tracking error~$e$ evolves in, if any. While $\|e_r(t)\|<\psi(t)$ is enforced by the control design, a boundary for $\|e(t)\|$ is not obvious. This is clarified by the following lemma.

\begin{Lem}\label{lem:epsi}
Let $e\in C^{r-1}([0,\omega),\R^m)$, $\omega\in(0,\infty]$, and consider the signals $e_1(t) = e(t)$ and $e_{i+1}(t) = \dot e_i(t) + k_i e_i(t)$, $k_i>0$, for $i=1,\ldots,r-1$. Further let $\psi\in C^1([0,\omega),\R)$ be such that $\psi(t)>0$ and $\dot \psi(t) \ge -\alpha \psi(t)$ for all $t\in[0,\omega)$, where $0\le \alpha<\min_{i=1,\ldots,r-1} k_i$. If $\|e_r(t)\| < \psi(t)$ for all $t\in [0,\omega)$, then we have for $i=1,\ldots,r-1$ and all $t\in [0,\omega)$:
\begin{multline}\label{eq:est-ki}
 \|e_i(t)\|
< \max\Bigg\{ \Bigg(\prod_{j=i}^{r-1} \frac{1}{k_j-\alpha}\Bigg) \\
\max_{j=i,\ldots,r-1} \Bigg(\prod_{p=r-j+i}^{r-1} \frac{1}{k_p-\alpha}\Bigg) \frac{\|e_j(0)\|}{\psi(0)}\Bigg\}\ \psi(t).
\end{multline}
\end{Lem}
\begin{pf} Define $c_r:=1$ and for $i=r-1,\ldots,1$ the constants  $c_i:= \max\left\{\frac{\|e_i(0)\|}{\psi(0)}, \frac{c_{i+1}}{k_i-\alpha}\right\}$, which are well defined since $k_i>\alpha$. We show that $\|e_i(t)\|\le c_i \psi(t)$ for all $t\in[0,\omega)$ by induction on~$i$. For $i=r$ the statement is clear and for $i=r-1,\ldots,1$ we calculate that
\begin{align*}
   & \tfrac12 \ddt \tfrac{\|e_i(t)\|^2}{\psi(t)^2} = -\tfrac{\dot \psi(t)}{\psi(t)} \tfrac{\|e_i(t)\|^2}{\psi(t)^2}   + \tfrac{1}{\psi(t)^2} e_i(t)^\top \big(e_{i+1}(t) - k_i e_i(t)\big)\\
    &\le (\alpha - k_i) \tfrac{\|e_i(t)\|^2}{\psi(t)^2} + \tfrac{\|e_{i+1}(t)\|}{\psi(t)} \tfrac{\|e_i(t)\|}{\psi(t)}\\
    &= \left((\alpha - k_i) \tfrac{\|e_i(t)\|}{\psi(t)} + c_{i+1}\right)  \tfrac{\|e_i(t)\|}{\psi(t)}
\end{align*}
for all $t\in[0,\omega)$. 
Then it follows from the comparison principle  that $\|e_i(t)\|\le c_i \psi(t)$ for all $t\in[0,\omega)$.
The assertion then follows from an additional straightforward induction on $i=r-1,\ldots,1$.\hfill $\Box$
\end{pf}

\section{Funnel control -- main result}\label{Sec:Main}
%

In this section we show that the application of the funnel controller~\eqref{eq:ICFC} to a system~\eqref{eq:Sys} under input constraints~\eqref{eq:IC} leads to a closed-loop initial-value problem which has a global solution.  By a solution of~\eqref{eq:Sys},~\eqref{eq:IC},~\eqref{eq:ICFC} on $[-h,\omega)$ we mean a pair of functions $(y,\psi)\in C^{r-1}([-h,\omega),\R^m)\times C([-h,\omega),\R)$ with $\omega\in (0,\infty]$, which satisfies $y|_{[-h,0]} = y^0$, $\psi(0)=\psi^0$  and $(y^{(r-1)},\psi)|_{[0,\omega)}$ is locally absolutely continuous and
satisfies the differential equations in~\eqref{eq:Sys} and~\eqref{eq:ICFC} with $u$ defined by~\eqref{eq:IC},~\eqref{eq:ICFC} for almost all $t\in[0,\omega)$;
$(y,\psi)$ is called maximal, if it has no right extension that is also a solution.\\
Next we present the main result of the present paper.

\begin{Thm}\label{Thm:FunCon}
Consider a system~\eqref{eq:Sys} with $(d,f,T)\in\cN^{m,r}$ for $m,r\in\N$, under input constraints~\eqref{eq:IC} with saturation function $\sat$ that satisfies~(P5). Let $y^0\in C^{r-1}([-h,0],\R^m)$ be the initial trajectory, $y_{\rm ref}\in W^{r,\infty}(\R_{\ge 0},\R^m)$ the reference signal and choose funnel control design parameters as in~\eqref{eq:FC-param}. Set $e=y-y_{\rm ref}$ and assume that the instantaneous values $e(t), \dot e(t),\ldots, e^{(r-1)}(t)$ are available for feedback and satisfy $\|e_r(0)\|<\psi^0$ for $e_r$ defined in~\eqref{eq:ICFC}. Then the funnel controller~\eqref{eq:ICFC} applied to~\eqref{eq:Sys},~\eqref{eq:IC} yields an initial-value problem which has
a solution, every solution can be maximally extended and every maximal solution $(y,\psi):[-h,\omega)\to\R^{m+1}$, $\omega\in(0,\infty]$, has the following properties:
\begin{enumerate}
  \item global existence: $\omega = \infty$;
  \item the functions $e_1,\ldots,e_{r-1}$ satisfy~\eqref{eq:est-ki} and $e_r$ satisfies 
  \[
    \exists\, \eps\in(0,1)\ \forall\, t\ge 0:\     \|e_r(t)\| \le  \eps \psi(t),
  \]
  in particular, $k$ and $v$ are bounded;
  \item if the saturation is not active on some interval $[t_0,t_1)\subseteq\R_{\ge 0}$ with $t_1\in (t_0,\infty]$, i.e., $v(t) = \sat(v(t))$ for all $t\in [t_0,t_1)$, then the performance funnel~$\psi$ reverts to its prescribed shape exponentially fast:
      \[
        \forall\, t\in [t_0,t_1):\ \psi(t) \le \tfrac{\beta}{\alpha} + \mu(t_0) e^{-\alpha (t-t_0)},
      \]
      where $\mu(t_0) := \psi(t_0) -  \tfrac{\beta}{\alpha}$.
\end{enumerate}
\end{Thm}
\begin{pf}
The proof consists of several steps.\\
\emph{Step 1}: We recast the closed-loop system in the form of an initial-value problem to which a well-known existence theory applies. First define the polynomial $p(s) = (s+k_{r-1})\cdots(s+k_1)$ and observe that $p(s) = s^{r-1} + \sum_{i=1}^{r-1} \mu_i s^{i-1}$ for some $\mu_1,\ldots,\mu_{r-1}>0$; set $\mu_r:=1$. Define the non-empty and relatively open set
\begin{align*}
&& \cD :=\left.\bigg\{(t,\xi_1,\ldots,\xi_r,\zeta)\in\R_{\ge 0}\times\left(\R^{m}\right)^r\times \R\  \right|\qquad \\
&& \left\|\sum\nolimits_{i=1}^{r} \mu_i \xi_i - p(\ddt)y_{\rm ref}(t)\right\|< \zeta,\ \zeta > \tfrac{\beta}{\alpha}\bigg\},
\end{align*}
and the functions
\begin{align*}
    &E:\cD\to\R^m,\ (t,\xi_1,\ldots,\xi_r,\zeta)\mapsto \sum_{i=1}^{r} \mu_i \xi_i - p(\ddt)y_{\rm ref}(t),\\
    &V:\cD\to\R^m,\ (t,\xi_1,\ldots,\xi_r,\zeta)\\
    &\quad \mapsto N\left(\tfrac{1}{1-\tfrac{\|E(t,\xi_1,\ldots,\xi_r,\zeta)\|^2}{\zeta^2}}\right)E(t,\xi_1,\ldots,\xi_r,\zeta).
\end{align*}
Further define
\begin{align*}
  &F\colon \cD\times\R^q\to\R^{rm+1},\ (t,\xi_1,\ldots,\xi_r,\zeta,\eta)\\
  &\quad \mapsto \begin{pmatrix} \xi_2\\\vdots\\\xi_r\\f\big(d(t),\eta,\sat(V(t,\xi_1,\ldots,\xi_r,\zeta))\big)\\
   -\alpha \zeta +\beta + \zeta \frac{\|V(t,\xi_1,\ldots,\xi_r,\zeta)-\sat(V(t,\xi_1,\ldots,\xi_r,\zeta))\|}{\|E(t,\xi_1,\ldots,\xi_r,\zeta)\|}
   \end{pmatrix}.
\end{align*}
Note that the function~$F$, and in particular its last component, is well-defined on $\cD\times\R^q$: Since $N$ is continuous and $\zeta>\tfrac{\beta}{\alpha}$, there exists $\delta>0$ such that for all $(t,\xi_1,\ldots,\xi_r,\zeta)\in\cD$ with $\|E(t,\xi_1,\ldots,\xi_r,\zeta)\|< \delta$ we have that $\|V(t,\xi_1,\ldots,\xi_r,\zeta)\| < \theta$ for $\theta$ as in~(P5), and hence $\|V(t,\xi_1,\ldots,\xi_r,\zeta)-\sat(V(t,\xi_1,\ldots,\xi_r,\zeta))\| = 0$.\\
Writing
\[
x(t)=\big(y(t)^\top,\ldots,y^{(r-1)}(t)^\top, \psi(t)\big)
\]
we see that the closed-loop initial-value problem~\eqref{eq:Sys},~\eqref{eq:IC},~\eqref{eq:ICFC}
may now be formulated as
\begin{equation}\label{eq:IVP-CL}
\begin{aligned}
\dot x(t)&=F\big(t,x(t),T(x)(t)\big),\\
x|_{[-h,0]}&=x^0\in C([-h,0],\R^{rm+1}),
\end{aligned}
\end{equation}
where, for $t\in [-h,0]$,
\[
x^0(t):=\big( y^0(t)^\top,\ldots,(y^0)^{(r-1)}(t)^\top, \psi^0\big)^\top.
\]
The function $F$ is measurable in~$t$, continuous in~$(\xi_1,\ldots,\xi_r,\zeta,\eta)$ and locally essentially bounded. By the assumptions $\|e_r(0)\|<\psi^0$ and $\psi^0>\tfrac{\beta}{\alpha}$ we see that $(0,x^0(0))\in\cD$. Therefore, an application of a variant of~\cite[Thm.~B.1]{IlchRyan09}\footnote{Although the property~(P3) of the operator~$T$ is weaker than required in~\cite{IlchRyan09}, this ``local'' property suffices for the proof.} yields the existence of a solution of~\eqref{eq:IVP-CL} and every solution can be extended to a maximal solution. Furthermore, any maximal solution $x:[-h,\omega)\to\R^{n}$, $\omega\in(0,\infty]$, of~\eqref{eq:IVP-CL} has the property that its graph
\[
    \cG := \setdef{(t,x(t))}{t\in[0,\omega)} \subset \cD
\]
has a closure which is not a compact subset of~$\cD$.\\
\emph{Step 2}: In this step we record some observations for later use. For $t\in[0,\omega)$, define $e_1(t) := e(t) = y(t) - y_{\rm ref}(t)$ for $y(t) = (x_1(t),\ldots,x_m(t))^\top$, and $e_{i+1}(t) := \dot e_i(t) + k_i e_i(t)$ for $i=1,\ldots,r-1$ as well as $\psi(t):= x_{rm+1}(t)$. Then $e_r(t) = E(t,x(t))$  and we have $\|e_r(t)\|<\psi(t)$, by which $k(t):= \left(1- \frac{\|e_r(t)\|^2}{\psi(t)^2}\right)^{-1}$ is well defined. Thus we arrive at the quantities in the control law~\eqref{eq:ICFC}; in particular $V(t,x(t)) = N(k(t)) e_r(t) = v(t)$. Furthermore, invoking $p(s)$ from Step~1, it follows that for almost all $t\in[0,\omega)$ we have
\begin{equation}\label{eq:ODE-ei}
  \dot e_r(t) = e^{(r)}(t) + \sum_{i=1}^{r-1} \mu_i e^{(i)}(t).
\end{equation}
Furthermore, since $\psi(t)>0$ it follows $\dot \psi(t) \ge -\alpha \psi(t) + \beta$ and hence $\psi(t) \ge \mu(0) e^{-\alpha t} + \tfrac{\beta}{\alpha}$ for all $t\in [0,\omega)$, where $\mu(\cdot)$ is defined in statement~(iii).\\
\emph{Step 3}: Fix $i\in\{0,\ldots,r-1\}$. We show that there exists $\kappa_i>0$ such that $\|y^{(i)}(t)\|\le \kappa_i \psi(t)$ for all $t\in [0,\omega)$. Observe that a straightforward induction utilizing $e_{i+1}(t) = \dot e_i(t) + k_i e_i(t)$ and $e_1(t) = e(t)$  gives that
\[
    e^{(i)}(t) = e_{i+1}(t) - \sum_{j=1}^i k_j e_j^{(i-j)}(t) = e_{i+1}(t) + \sum_{j=1}^i c_{i,j} e_j(t)
\]
for some $c_{i,j}\in\R$, $i=1,\ldots,r-1$, $j=1,\ldots,i$. Invoking $\|e_r(t)\|<\psi(t)$ and Lemma~\ref{lem:epsi} (notice that~$\psi$ satisfies the assumptions of the lemma by the observations in Step~2), it follows that~\eqref{eq:est-ki} holds, i.e., there exist $\sigma_1,\ldots,\sigma_{r-1}$ such that, with $\sigma_r:=1$, we have $\|e_i(t)\|\le \sigma_i \psi(t)$ for all $t\in[0,\omega)$ and $i=1,\ldots,r$. Then 
\begin{align*}
    \|y^{(i)}(t)\|&\le \|e_{i+1}(t)\| + \sum_{j=1}^i |c_{i,j}| \|e_j(t)\| + \|y_{\rm ref}^{(i)}(t)\| \\
    &\le \bigg( \sigma_{i+1} +  \sum_{j=1}^i |c_{i,j}| \sigma_j + \frac{\alpha}{\beta} \|y_{\rm ref}^{(i)}\|_\infty\bigg) \psi(t)
\end{align*}
for all $t\in[0,\omega)$, where we used $1\le \tfrac{\alpha}{\beta} \psi(t)$ by Step~2. This proves the assertion for $\kappa_i:= \left( \sigma_{i+1} +  \sum_{j=1}^i |c_{i,j}| \sigma_j + \frac{\alpha}{\beta} \|y_{\rm ref}^{(i)}\|_\infty\right)$.\\
\emph{Step 4}: We show that there exists $C>0$ such that for all $t\in [0,\omega)$ we have
\[
    \|f\big(d(t),T(y,\dot y, \ldots,y^{(r-1)})(t),\sat(v(t))\big)\| \le C \psi(t).
\]
By the sector bound property~(P4) and Step~3 we have that
\begin{align*}
   & \|f\big(d(t),T(y,\dot y, \ldots,y^{(r-1)})(t),\sat(v(t))\big)\| \\
    & \le M_{1}\big(t,d(t),\sat(v(t))\big) \\
    &\quad + \sum_{i=1}^{r} M_{i+1}\big(t,d(t),\sat(v(t))\big) \|y^{(i-1)}|_{[-h,t]}\|_\infty\\
    & \le M_{1}\big(t,d(t),\sat(v(t))\big) \frac{\alpha}{\beta} \psi(t) \\
    &\quad + \sum_{i=1}^{r} M_{i+1}\big(t,d(t),\sat(v(t))\big) \big(\|(y^0)^{(i-1)}\|_\infty \!+\! \|y^{(i-1)}\|_\infty\big).
\end{align*}
for all $t\in [0,\omega)$. Since $M_1,\ldots,M_{r+1}$ are continuous and bounded in~$t$ and $d$, $\sat(v)$ are bounded, there exist $\overline M_1,\ldots, \overline M_{r+1}>0$ such that $\|M_{i}\big(t,d(t),\sat(v(t))\big)\|\le \overline M_i$ for all $t\in[0,\omega)$ and all $i=1,\ldots,r+1$. Then, by Step~3, the assertion holds for 
\begin{align*}
    C&:=  \tfrac{\alpha}{\beta} \overline M_{1}  + \sum_{i=1}^{r} \overline M_{i+1}\left(\tfrac{\alpha}{\beta} \|(y^0)^{(i-1)}\|_\infty + \kappa_{i-1}\right).
\end{align*}
\emph{Step 5}: We show that $k\in L^\infty([0,\omega),\R)$. Note that, invoking~\eqref{eq:ICFC} we may estimate
\begin{equation}\label{eq:est-kappa}
   \forall\, t\in [0,\omega):\ \kappa(v(t)) \ge |N(k(t))| \cdot \|e_r(t)\| - M,
\end{equation}
where $M>0$ is some upper bound of $\sat$, i.e., $\|\sat(v)\|\le M$ for all $v\in\R^m$. Now set 
\[
    \hat C:= C + \sum_{i=1}^r \mu_i \kappa_i + \frac{\alpha}{\beta} \bigg(\|y_{\rm ref}^{(r)}\|_\infty + \sum_{i=1}^{r-1} \mu_i \|y_{\rm ref}^{(i)}\|_\infty\bigg)
\]
and choose
\begin{equation}\label{eq:delta}
    \delta > \alpha + \hat C + \frac{\alpha}{\beta} M
\end{equation}
and $\eps\in (0,1)$ so that, invoking $\|e_r(0)\|<\psi^0$,
\begin{align*}
    \eps > \frac{\|e_r(0)\|}{\psi^0}\quad\text{and}\quad \eps \left| N\left(\frac{1}{1-\eps^2}\right)\right| \ge 2\delta,
\end{align*}
where the latter is possible because of the properties of~$N$ in~\eqref{eq:FC-param}. We show that $\|e_r(t)\| \le \eps \psi(t)$ for all $t\in [0,\omega)$, which is equivalent to $k\in L^\infty([0,\omega),\R)$. Seeking a contradiction, assume there exists $t_1\in [0,\omega)$ such that $\|e_r(t_1)\| > \eps \psi(t_1)$ and define
\[
    t_0 := \max\setdef{t\in[0,t_1)}{\|e_r(t)\|= \eps \psi(t)}.
\]
Then, for all $t\in[t_0,t_1]$, we have
\begin{align}\label{eq:est-er-kr}
  \|e_r(t)\| \ge \eps \psi(t) 
  \quad \text{and}\quad
  k(t)  \ge  \frac{1}{1-\eps^2}.
\end{align}
Since $|N(k(t_0))| = |N(\tfrac{1}{1-\eps^2})| \ge 2\delta/\eps$, there exists $t_2\in (t_0,t_1]$ such that
\[
   \forall\, t\in [t_0,t_2]:\ |N(k(t))|\ge \frac{\delta}{\eps}.
\]
Furthermore, by definition of~$t_0$ we have that $\|e_r(t_2)\| > \eps \psi(t_2)$. Then we obtain that
\begin{align*}
  &\tfrac12 \ddt \|e_r(t)\|^2 \stackrel{\eqref{eq:ODE-ei}}{=} e_r(t)^\top \bigg( e^{(r)}(t) + \sum_{i=1}^{r-1} \mu_i e^{(i)}(t)\bigg)\\
  &\quad\stackrel{\eqref{eq:Sys}}{\le}  \bigg( \|f\big(d(t),T(y,\dot y, \ldots,y^{(r-1)})(t),\sat(v(t))\big)\| \\
  &\qquad + \|y_{\rm ref}^{(r)}(t)\| + \sum_{i=1}^{r-1} \mu_i \|e^{(i)}(t)\|\bigg) \|e_r(t)\|\\
  &\ \stackrel{\rm Steps\,3,4}{\le}  \bigg( C \psi(t) + \|y_{\rm ref}^{(r)}\|_\infty \\
  &\quad + \sum_{i=1}^{r-1} \mu_i \big(\kappa_i \psi(t) + \|y_{\rm ref}^{(i)}\|_\infty\big)\bigg) \|e_r(t)\|\\
   &\ \stackrel{\rm Step\,2}{\le} \big(\eps \dot \psi(t) + \hat C \psi(t) - \eps \dot \psi(t) \big) \|e_r(t)\|\\  
    &\quad\stackrel{\eqref{eq:ICFC}}{=} \Big(\eps \dot \psi(t) \!+\! \hat C \psi(t) \!+\! \eps \alpha \psi(t) \!-\! \eps \beta \!-\! \eps \psi(t) \tfrac{\kappa(v(t))}{\|e_r(t)\|}\Big) \|e_r(t)\|\\
    &\stackrel{\eqref{eq:est-kappa},\eqref{eq:est-er-kr}}{\le} \Big(\eps \dot \psi(t) \!-\! \eps \beta \!+\! M \!-\! \big( \eps |N(k(t))| \!-\! \eps \alpha  \!-\!\hat C\big) \psi(t) \Big) \|e_r(t)\|\\
  &\quad \le \Big(\eps \dot \psi(t) + M - \underset{>0\ \text{by~\eqref{eq:delta}}}{\underbrace{\big( \delta - \alpha  -\hat C\big)}} \psi(t) \Big) \|e_r(t)\|\\
  &\ \stackrel{\rm Step\,2}{\le} \Big(\eps \dot \psi(t) + M  - \big( \delta - \alpha  -\hat C\big) \tfrac{\beta}{\alpha} \Big) \|e_r(t)\|\\
  &\quad \stackrel{\eqref{eq:delta}}{\le} \eps \dot \psi(t) \|e_r(t)\|
\end{align*}
for almost all $t\in[t_0,t_2]$ and upon integration we get
\begin{align*}
  \|e_r(t_2)\| - \|e_r(t_0)\| &= \int_{t_0}^{t_2}  \tfrac12 \|e_r(t)\|^{-1} \ddt \|e_r(t)\|^2 {\rm d}t\\
  &\le \int_{t_0}^{t_2}  \eps \dot \psi(t) {\rm d}t = \eps \psi(t_2) - \eps \psi(t_0),
\end{align*}
which yields the contradiction
\[
    0 = \eps \psi(t_0) - \|e_r(t_0)\|  \le \eps \psi(t_2) - \|e_r(t_2)\| < 0.
\]
\emph{Step 6}: We show that $\omega=\infty$, i.e., assertion~(i) of the theorem. Suppose that $\omega<\infty$. From Step~5 it follows that there exists $\eps\in (0,1)$ such that
\[
    \forall\, t\in[0,\omega):\ \|e_r(t)\|\le \eps \psi(t).
\]
Furthermore, by Step~2 we have $\psi(t)\ge \mu(0) e^{-\alpha \omega} + \tfrac{\beta}{\alpha} > \tfrac{\beta}{\alpha}$ for all $t\in[0,\omega)$. Moreover, from boundedness of $k$ it follows that~$\kappa(v)$ is bounded and hence $\tfrac{\kappa(v)}{\|e_r\|}$ is bounded, since $\kappa(v)$ vanishes when $\|e_r\|$ is small enough, cf.\ Step~1. Therefore, it follows from~\eqref{eq:ICFC} that there exist some $d_{1}, d_{2}\ge 0$ such that $\psi(t) \le d_{1} e^{d_{2} t} \le d_{1} e^{d_{2} \omega}$ for all $t\in [0,\omega)$. Define
\begin{align*}
&& \hat \cD :=\left.\bigg\{(t,\xi_1,\ldots,\xi_r,\zeta)\in[0,\omega]\times\left(\R^{m}\right)^r\times \R\  \right|\quad \\
&& \left\|\sum\nolimits_{i=1}^{r} \mu_i \xi_i - p(\ddt)y_{\rm ref}(t)\right\|\le \eps \zeta,\quad\\
&& \tfrac{\beta}{\alpha} + \mu(0) e^{-\alpha \omega} \le \zeta \le d_{1} e^{d_{2} \omega}\bigg\},
\end{align*}
which is evidently a compact subset of $\cD$ since $y_{\rm ref},\ldots,y_{\rm ref}^{(r-1)}$ are bounded. Since $(t,x(t))\in \hat \cD$ for all $t\in[0,\omega)$, it follows that the closure of the set~$\cG$ from Step~1 is a compact subset of~$\cD$, a contradiction. Therefore, $\omega=\infty$.\\
\emph{Step 7}: We complete the proof by establishing assertions~(ii) and~(iii) of the theorem. Assertion~(ii) is a consequence of Steps~3 and~5. Let  $[t_0,t_1)\subseteq\R_{\ge 0}$ with $t_1\in (t_0,\infty]$ be an interval with $v(t) = \sat(v(t))$ for all $t\in [t_0,t_1)$, then  statement~(iii) is clear since $\dot \psi(t) = -\alpha \psi(t) + \beta$ for all $t\in[t_0,t_1)$. This completes the proof.\hfill $\Box$
\end{pf}

We stress that for the global solution provided by Theorem~\ref{Thm:FunCon} the funnel function $\psi$ is not bounded in general. However, \textit{a posteriori}, by~(iii) the funnel boundary reverts to its prescribed shape on any interval where the saturation is not active; in particular, if $t_1=\infty$, then it is bounded.


\begin{Rem}
    We comment on the freedom of choice of the design parameters in~\eqref{eq:FC-param}. The parameters $\alpha, \beta$ and $\psi^0$ are chosen by the user to determine the desired shape of the funnel boundary in the form $\psi_{\rm des}(t) = \left(\psi^0-\tfrac{\beta}{\alpha}\right) e^{-\alpha t} + \tfrac{\beta}{\alpha}$. Then, according to~\eqref{eq:est-ki}, a suitable choice for $k_1,\ldots,k_{r-1}$ could be $k_i=\alpha+1$, resulting in the estimate $\|e_i(t)\|< \max_{j=i,\ldots,r-1} \left\{1, \tfrac{\|e_j(0)\|}{\psi^0}\right\}\psi(t)$ for all $t\ge 0$ and all $i=1,\ldots,r-1$. Finally, a typical choice for the surjection~$N$ is $N(s) = s\sin s$.\\
    Compared to the precursor of the improved controller~\eqref{eq:ICFC} presented in~\cite{Berg24}, it is another advantage that it has much fewer design parameters. For the precursor it is generally hard to determine suitable parameters.
\end{Rem}

\section{Simulations}\label{Sec:Sim}

We compare the controller~\eqref{eq:ICFC} to its precursor presented in~\cite{Berg24} and consider the benchmark example of the mass-on-car system presented therein, which is originally taken from~\cite{SeifBlaj13}. As shown in Fig.~\ref{Mass.on.car}, the mass~$m_2$ (in~\si{\kilo\gram})  moves on a ramp inclined by the angle~{$\vartheta \in [0,\frac{\pi}{2})$} (in \si{\radian}) and is mounted on a car with mass~$m_1$ (in \si{\kilo\gram}). The control input is the force~$u=F$ (in \si{\newton}) which acts on the car. The equations of motion for the system are given by
{\small
\begin{equation}\label{mass.on.car.equ}
\begin{bmatrix}
m_1+m_2&m_2\cos \vartheta\\
m_2\cos \vartheta&m_2
\end{bmatrix} \begin{pmatrix} \ddot{{ z}}(t)\\ \ddot{s}(t) \end{pmatrix} +\begin{pmatrix}
0\\
ks(t)\!+\!d\dot{s}(t)
\end{pmatrix}\!=\!\begin{pmatrix}
u(t)\\
0
\end{pmatrix},
\end{equation}
}
where $t$ is the current time (in \si{\second}),  $z$ (in \si{\metre}) is the horizontal car position and~$s$ (in \si{\metre}) the relative position of the mass on the ramp. The coefficients of the spring and damper are given by~$k  >0$ (in \si{\newton\per{\metre}}) and $d >0$ (in \si{\newton\second\per{\metre}}), resp. The output
 $y$ (in \si{\metre}) is the horizontal position of the mass on the ramp given by $y(t)={ z}(t)+s(t)\cos \vartheta$.
    \begin{figure}[htp]
    \begin{center}
    \includegraphics[trim=2cm 4cm 5cm 15cm,clip=true,width=5.2cm]{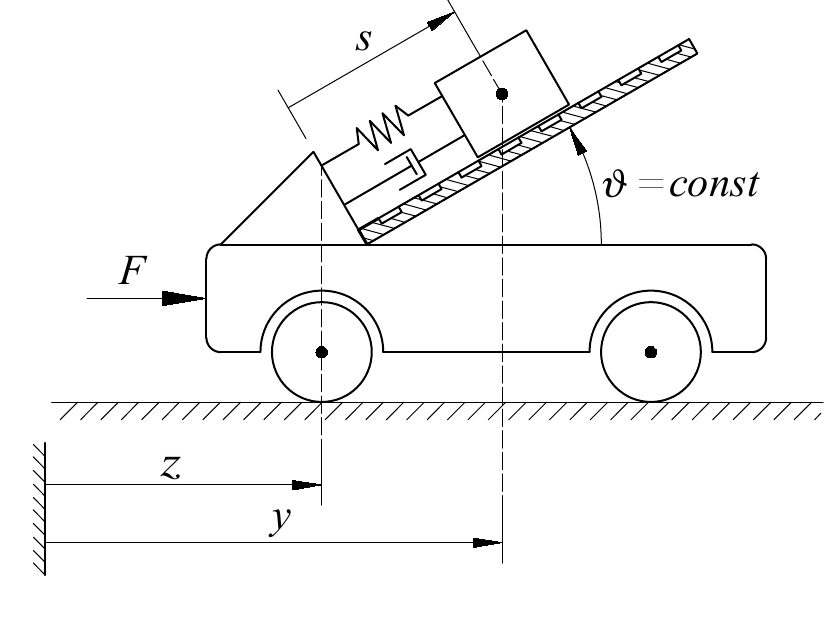}
    \end{center}
    \vspace*{-3mm}
    \caption{Mass-on-car system.}
    \label{Mass.on.car}
    \end{figure}
It can be observed that for $u=0$ the system admits the solution $t\mapsto z(t) := t$ and $t\mapsto s(t) := 0$, thus the system is not input-to-state stable and hence~\cite[Thm.~1]{TrakBech24} cannot be applied.\\
For the simulation we consider the case of order $r=3$, that is $\vartheta=0$. As shown in~\cite{Berg24}, system~\eqref{mass.on.car.equ} with output~$y$ belongs to the class $\cN^{1,3}$ in this case. We choose the parameters
$m_1=4$, $m_2=1$, $k=2$, $d=1$,
the initial values $z(0)= s(0) = 0$, $\dot{z}(0) = \dot s(0)= 0$ and the reference signal $y_{\rm ref} \colon t\mapsto \tfrac12 \cos t$. 
The saturation function in~\eqref{eq:IC} is chosen as $\sat(v)=v$ for $|v|\le M$ and $\sat(v) =\sgn(v) M$ for $|v|>M$, with $M=8$. All simulations are MATLAB generated (solver: {\tt ode45}, rel.\ tol.: $10^{-10}$, abs.\ tol.: $10^{-8}$) and over the time interval $[0,20]$.\\
For both controllers~\eqref{eq:ICFC} and its precursor from~\cite{Berg24}, we choose the common parameters $N(s)=s\sin s$ and $\alpha=\alpha_1=1.5$, $\beta=\beta_1=0.15$, $\psi^0=\psi_1^0=3.1$, so that the desired funnel boundary is $\psi_{\rm des}(t) = 3e^{-1.5 t} + 0.1$. Furthermore, for~\eqref{eq:ICFC} we choose $k_1=k_2=\alpha+1$, and for the controller from~\cite{Berg24} we choose
\begin{align*}
    &\alpha_2 = 0.9\alpha_1,\ \alpha_3 = 0.9\alpha_2,\ \beta_2 = 0.5\alpha_2,\ \beta_3=0.5\alpha_3,\\
    &p_1=p_2=1.1,\ \psi_2^0=2,\ \psi_3^0=1.
\end{align*}
\captionsetup[subfloat]{labelformat=empty}
\begin{figure}[h!tb]
  \centering
  \subfloat[Fig.~\ref{fig:sim}a: Performance funnels and tracking errors]
{
\centering
  \includegraphics[width=0.39\textwidth]{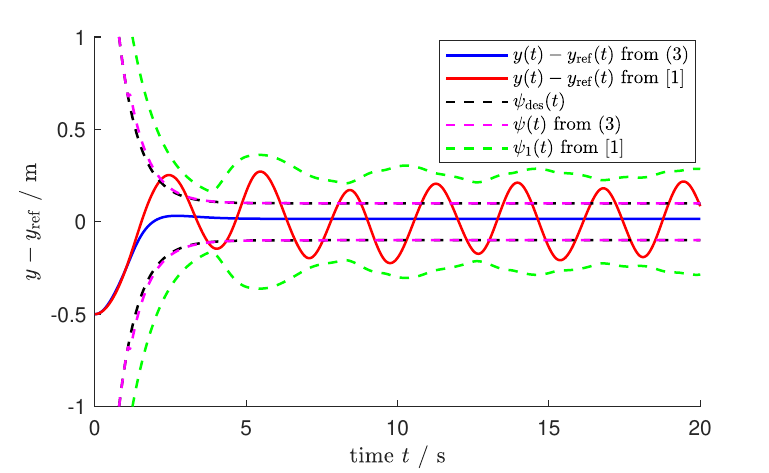}
\label{fig:sim-e}
}\\
\subfloat[Fig.~\ref{fig:sim}b: Input functions]
{
\centering
\includegraphics[width=0.39\textwidth]{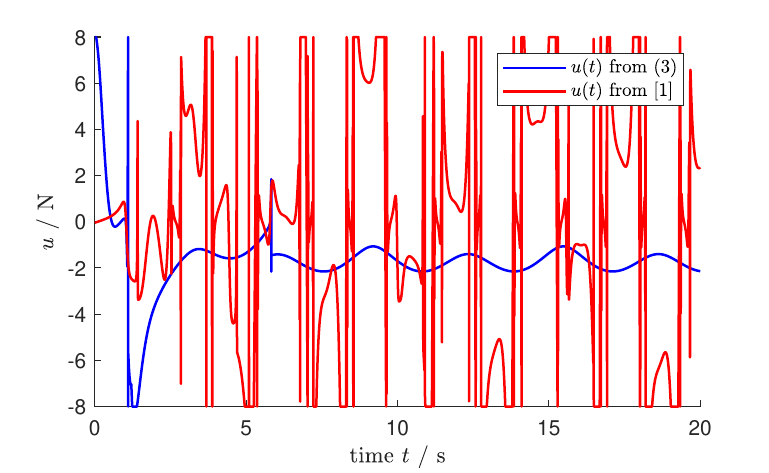}
\label{fig:sim-u}
}
\caption{Simulation of controller~\eqref{eq:ICFC} and its precursor from~\cite{Berg24}, labeled as [1].} 
\label{fig:sim}
\end{figure}
The application of the controller~\eqref{eq:ICFC} and its precursor from~\cite{Berg24} to~\eqref{mass.on.car.equ} is depicted in Fig.~\ref{fig:sim}. The corresponding tracking errors and funnel boundaries are shown in
Fig.~\ref{fig:sim-e}, while Fig.~\ref{fig:sim-u} shows the respective input functions. It is evident that the performance of the improved controller~\eqref{eq:ICFC} is superior to its precursor. While the latter constantly induces periods of active saturation with a tracking error frequently leaving the desired performance funnel, the improvement~\eqref{eq:ICFC} only saturates over a short period at the beginning and quickly drives the tracking error to zero (never leaving the desired funnel), even under tight input constraints. 

\vspace*{-1mm}







\bibliography{MST-TB,references}



\end{document}